\documentclass[a4paper,11pt]{article}
\usepackage[sc]{mathpazo}
\usepackage[T1]{fontenc}
\usepackage{youngtab}
\usepackage{amsthm,amsmath, amsfonts,hyperref,a4wide,amssymb}
\usepackage[usenames,dvipsnames]{xcolor} 
\usepackage{tikz}
\makeatletter
\tikzset{my loop/.style =  {to path={
  \pgfextra{}
  [looseness=12,min distance=6mm]
  \tikz@to@curve@path},font=\sffamily\small
  }}  
\makeatletter 
\usetikzlibrary{matrix,arrows,backgrounds,positioning,fit,shapes}

\newtheorem{theorem}{Theorem}
\newtheorem*{theorem*}{Theorem}
\newtheorem{lemma}[theorem]{Lemma}

\newtheorem{proposition}[theorem]{Proposition}
\theoremstyle{definition}

\newtheorem{ex}[theorem]{Example}
\newtheorem{df}[theorem]{Definition}
\theoremstyle{remark}
\newtheorem{rem}{Remark}

\newcommand*{\R}{\mathbb{R}}
\newcommand*{\Z}{\mathbb{Z}}

\newcommand*{\N}{\mathbb{N}}
\newcommand*{\C}{\mathbb{C}}

\newcommand*{\F}{\mathcal{F}}
\newcommand*{\G}{\mathcal{G}}

\newcommand*{\rk}{\mathrm{rk}}

\newcommand*{\orth}{\text{O}}
\newcommand*{\spl}{\text{Sp}}

\newcommand{\la}{\langle}
\newcommand{\ra}{\rangle}

\title{Mixed partition functions and exponentially bounded edge-connection rank}
\author{Guus Regts\thanks{University of Amsterdam, Sciencepark 105-107, 1098 XH Amsterdam, the Netherlands. Email: \texttt{guusregts@gmail.com.} Supported by a NWO Veni grant.}\and Bart Sevenster \thanks{University of Amsterdam. Email: \texttt{blsevenster@gmail.com.} Supported by ERC grant number 339109. The present paper is partly based on this author's PhD thesis~\cite{thesis Bart}.}}

\begin{document}
\maketitle

\begin{abstract}
\noindent 
We study graph parameters whose associated edge-connection matrices have exponentially bounded rank growth.
Our main result\footnote{Some of the results of this paper were announced in an extended abstract~\cite{RS17a}. Unfortunately \cite{RS17a} contains a mistake; we will comment on that below.}. is an explicit construction of a large class of graph parameters with this property that we call \emph{mixed partition functions}. 
Mixed partition functions can be seen as a generalization of partition functions of vertex models, as introduced by de la Harpe and Jones, [P. de la Harpe, V.F.R. Jones, Graph invariants related to statistical mechanical models:
examples and problems, {\sl Journal of Combinatorial Theory}, Series B {\bf 57} (1993) 207--227] and they are related to invariant theory of orthosymplectic supergroup.
We moreover show that evaluations of the characteristic polynomial of a simple graph are examples of mixed partition functions, answering a question of de la Harpe and Jones.
\\\quad 


\noindent \textbf{M.S.C [2010]}  Primary   05C45, 15A72; Secondary 05C25, 05C31
\end{abstract}

\section{Introduction}
De la Harpe and Jones \cite{HJ} introduced vertex models on graphs as a generalization of the Ising-Potts model of statistical physics. Partition functions of vertex models form a rich class of graph parameters including, for instance, the number of matchings, the permanent of the adjacency matrix of a simple graph and the number of graph homomorphisms into a fixed graph \cite{Sz7}. Closely related concepts appear in quantum information theory as \emph{tensor network contractions}~\cite{MS08}, in theoretical computer science as \emph{Holant problems}~\cite{CLX11} and in knot theory as \emph{Lie algebra weight systems}~\cite{CDM12}.
In the combinatorics literature a vertex model is often called an edge-coloring model, cf. \cite{Sz7,L12} and we adopt this terminology here.



\paragraph{Partition functions of edge coloring models}
Before we give the definition of partition functions of edge-coloring models, we first introduce some useful terminology.
All graphs considered are finite and may have loops or multiple edges. The set of all graphs is denoted by $\G$. 
We will consider a single edge with an empty vertex set as a connected graph and we call this graph the \emph{circle} and denote it by $\bigcirc$.
For the purposes of this paper we let $\N=\mathbb{Z}_{\geq 0}$.

We next turn to the definition of partition function of edge-coloring models. We do this in a slightly unconventional way. This will however turn out to be convenient later.
Let $k \in \N$ and let $V_k$ be a vector space over $\C$ with basis $\{e_1,\ldots,e_k\}$. Let $S V_k$ be the symmetric algebra of $V_k$. Recall that the symmetric algebra is the quotient of the tensor algebra $T V_k$ by the ideal generated by $\{ x \otimes y - y \otimes x \mid x,y \in V_k \}$. The image of $e_{i_1} \otimes \cdots \otimes e_{i_n}$ in $S V_k$ under this quotient map is denoted by $e_{i_1} \odot \cdots \odot e_{i_n}$. 
A basis of $S V_k$ is given by the elements $\{\bigodot_{i \in S} e_i\}$, where $S$ is a multiset consisting of elements in $[k]:=\{1,\ldots,k\}$. An element $h \in  (S V_k)^*$ is called a \emph{$k$-color edge coloring model}\index{edge coloring model!ordinary}. The \emph{partition function}\index{partition function!ordinary} $p_h$ of the edge coloring model $h$ is the graph parameter over $\C$ defined, for a graph $G = (V,E)$, by
\begin{equation}\label{eq:pf}
p_h(G) := \sum_{\phi : E \rightarrow [k]} \prod_{v \in V} h(\bigodot_{a \in \delta(v)} e_{\phi(a)}),
\end{equation}
where $\delta(v)$ is the set of edges incident with $v$. Note that $p_h(\bigcirc) = k$. If a graph parameter is the partition function of a $k$-color edge coloring model for some $k \in \N$, then we will refer to it as an \emph{ordinary partition function}\index{partition function!ordinary}.  A graph parameter $f$ is called \emph{multiplicative} if for two graphs $G$ and $H$ we have that $f(G \cup H) = f(G) f(H)$. We note that the partition function of an edge-coloring model $h$ is multiplicative.


\begin{ex}
Let us show that the number of matchings in a graph is an ordinary partition function. 
Take $k=2$. Let $h \in (S V_2)^*$ be defined by $h(e_1^{\odot n_1} \odot e_2^{\odot n_2})=1$ if $n_2\leq 1$ and zero otherwise.
Then the set of edges mapped to $2$ by a map $\phi:E\to [2]$ form a matching if and only if each vertex sees color $2$ at most once. Hence these edges form a matching if and only if $\phi$ contributes $1$ to the sum \eqref{eq:pf}.
\end{ex}

We have just defined (ordinary) partition functions for graphs, but they can also be defined for directed graphs \cite{HJ}, virtual link diagrams \cite{RSS17}, chord diagrams \cite{S15a}, and several other combinatorial structures \cite{S15b}.
Instead of coloring the edges, one can also color the vertices. 
This then gives rise to the partition function of a \emph{spin model} cf. \cite{HJ}, or \emph{vertex coloring model} cf.~\cite{FLS,L12}. See Subsection~\ref{sec:spin} for details.

\paragraph{Edge connection matrices}
Several types of connection matrices have played an important role in characterizing partition functions~\cite{FLS,Sz7,S15,S15a,S15b,RSS16,RSS17}. 
Connection matrices also play an important role in definability of graph parameters~\cite{KM12}. 
We refer to the book of Lov\'asz~\cite{L12} and to \cite{R13} for further background. 

The connection matrices we consider in this paper originate in \cite{Sz7}. 
To define them we need the concept of a fragment.
For $t\in \N$, a \emph{$t$-fragment} is a graph which has $t$ vertices of degree one labeled $1,2,\ldots,t$. 
We will refer to an edge incident with a labeled vertex as an \emph{open end}. 
Each open end is labeled with the label of the labeled vertex incident with it.
Let $\F_t$ denote the collection of all $t$-fragments. 
Then $\F_0$ can be considered as the collection of all graphs $\G$.
For $i=1,2$, let $F_i$ be a $t_i$-fragment and let $o_i$ be an open end of $F_i$. 
The procedure of removing the labeled vertices incident to the $o_i$ and gluing the edges $o_1$ and $o_2$ together is referred to as the \emph{gluing} of $o_1$ and $o_2$. This way a $t_1+t_2-2$-fragment is obtained.
Define a \emph{gluing operation} $*:\F_t\times \F_t\to \G$, where, for two $t$-fragments $F_1,F_2$, $F_1*F_2$ is the graph obtained from $F_1\cup F_2$ by gluing the open ends with the same labels together.
Note that by gluing two edges of which both the endpoints are labeled one creates a circle.
For a graph parameter $f:\G\to \C$, the $t$-th \emph{edge-connection matrix} $M_{f,t}$ is an infinite matrix whose rows and columns are indexed by $t$-fragments and which is defined by, for $F_1,F_2\in \F_t$,
\[
M_{f,t}(F_1,F_2):=f(F_1*F_2).
\]

\paragraph{Edge-rank connectivity}
We say that a graph parameter $f:\G\to \C$ has \emph{exponentially bounded edge-connection rank} if there exists a constant $r>0$ such that $\rk (M_{f,t})\leq r^t$ for all $t\in \N$.
Following~\cite{FLS}, the minimal such $r$ we call the \emph{edge-rank connectivity} of $f$. 
By Fekete's Lemma this minimal $r$ is equal to $\text{sup}_{t \in \Z_+} (\text{rk}(M_{f,t}))^{1/t}$. 
This definition is additionally motivated from an algorithmic perspective. Using a variation of an idea of Lov\'asz and Welsh (see \cite[Theorem 6.48]{L12}) one can show that a graph parameter $f$ with edge-rank connectivity $r$ can be computed in time $r^{O(\Delta(G)\text{tw}(G))}(|V(G)|)^{O(1)}$ for any graph $G$. 
(Here $\Delta(G)$ denotes the maximum degree, $\text{tw}(G)$, the treewidth, see e.g.~\cite{Dies}, and $V(G)$ the vertex set of $G$.).

It is not difficult to show that if $f$ is the partition function of a $k$-color edge coloring model, then the edge-rank connectivity of $f$ is at most $k$. In fact, Schrijver \cite{S15} characterized partition functions of edge-coloring models over $\C$ using the edge-rank connectivity:
\begin{theorem}[\cite{S15}]\label{eq:thm lex}
A graph parameter $f:\G\to \C$ is the partition function of an edge-coloring model over $\C$ if and only if
\begin{equation}
f(\emptyset)=1, f(\bigcirc)\in \R \text{ and }\text{rk}(M_{f,t})\leq f(\bigcirc)^t\text{ for all }t\in \N.
\end{equation}
\end{theorem}
In~\cite{DGLRS12} an example of a graph parameter with finite edge rank connectivity was given that is not the partition function of an edge-coloring model. 
In \cite{RS17}, the authors of the present paper generalized this example by introducing skew-partition functions and showing that these have finite edge-connection rank connectivity.

\paragraph{Our contributions}
In the present paper we introduce a class of graph parameters that may be seen as a mix of ordinary partition functions and skew-partition functions, that we call \emph{mixed partition functions}.
Their definition is a little involved and will be given in Section~\ref{sec:def}.
Mixed partition functions are connected to invariant theory of the orthogonal group, the symplectic group and the orthosymplectic Lie super algebra.\footnote{In the extended abstract~\cite{RS17a} we incorrectly announced that the orthosymplectic Lie super algebra does not play a role.} 
While we say a few words about this connection in Section~\ref{sec:invariant}, we will focus on combinatorial aspects of mixed partition functions in the present paper and consider the connection to invariant theory in a subsequent paper~\cite{RS18}. 

One of our main results is that mixed partition functions have finite edge-rank connectivity and we in fact conjecture that mixed partition functions form the entire class of graph parameters (that take value $1$ on the empty graph) of finite edge-rank connectivity\footnote{In~\cite{RS17a} we in fact announced this to be true. We hope to be able to prove this conjecture in~\cite{RS18}. As there are quite a number of technical challenges related to supersymmetry and the fact that the orthosymplectic Lie super algebra is generally not semisimple, we choose to call it a conjecture for now.}.
We shall give various examples in Section~\ref{sec:examples} of mixed partition functions. In particular we shall show that a certain evaluation of the characteristic polynomial of simple graphs cannot be realized as the partition function of  an edge-coloring model, answering a question of de la Harpe and Jones~\cite{HJ}, while we show that any evaluation of the characteristic polynomial can be described as a mixed partition function.

\section{Definitions}\label{sec:def}
We now turn to the definition of the relevant edge-coloring models and their mixed partition functions. 
We first recall the construction of skew-partition functions from \cite{RS17}, which we modify slightly. 

\subsection{Skew partition functions}
An \emph{Eulerian graph}\index{Eulerian graph} $G = (V,E)$ is a graph such that each vertex has even degree (we do not assume $G$ to be connected). A \emph{local pairing}\index{local pairing} $\kappa$ of $G$ is at each vertex $v \in V$ a decomposition $\kappa_v$ of the edges incident with $v$ into ordered pairs, i.e., $\kappa_v = \{(a_1,a_2),\ldots,(a_{d(v)-1},a_{d(v)}) \}$. If $(a_i,a_{i+1}) \in \kappa_v$, then we say that $a_i$ and $a_{i+1}$ are \emph{paired}\index{paired} at $v$. 
Let $\omega$ be an Eulerian orientation of the edges of $G$ and $\kappa$ a local pairing of $G$. We say that $\kappa$ and $\omega$ are \emph{compatible} if for each vertex $v$ and for each $(a_1,a_2) \in \kappa_v$ the arc $a_1$ is incoming at $v$ and the arc $a_2$ is outgoing at $v$ under $\omega$.
 
Recall that a circuit is a closed walk where vertices may occur multiple times, but edges may not. Fix an Eulerian orientation $\omega$. A compatible local pairing $\kappa$ decomposes $E$ into circuits that, after choosing a starting point $v_0$ and a direction, are of the form $(v_0,e_1,\ldots,e_i,v_i,e_{i+1},\ldots,v_n)$, where $v_0 = v_n$ and such that $e_i$ and $e_{i+1}$ are paired at $v_i$ for each $i\in[n]$. We consider the indices modulo $n$. We will refer to a circuit in this decomposition as a \emph{$\kappa$-circuit}. We define $c(\kappa)$ to be the number of circuits in this decomposition. 

Let $V_{2 \ell}:=\C^{2\ell}$ and let $\{f_1,\ldots,f_{2\ell}\}$ be the standard basis of $V_{2 \ell}$. The \emph{exterior algebra} $\bigwedge V_{2\ell}$ of $V_{2\ell}$ is the quotient of the tensor algebra $T V_{2\ell}$ by the ideal generated by $\{x \otimes y + y \otimes x \mid x,y \in V_{2\ell} \}$. The image of $f_{i_1} \otimes \cdots \otimes f_{i_n}$ under this quotient map is denoted by $f_{i_1} \wedge \cdots \wedge f_{i_n}$. 

Define $g_i\in V_{2 \ell}$ for $i=1,\ldots,2\ell$ by 
\begin{equation}\label{eq:dual}
g_i:=\left\{\begin{array}{rl} -f_{i+\ell} &\text{ if } i\leq \ell,
\\ f_{i-\ell}& \text{ if }i>\ell.
\end{array}\right.
\end{equation}

Let $G = (V,E)$ be an Eulerian graph with Eulerian orientation $\omega$ and compatible local pairing $\kappa$.  
For $h=(h_v)_{v\in V}$ with $h^v \in (\bigwedge V_{2 \ell})^*$ for each $v\in V$, we define
 \begin{equation}\label{eq:function sh}
s_h(G,\omega,\kappa):=(-1)^{c(\kappa)}\sum_{\phi:E\to [2\ell]}\prod_{v\in V}h^v(\bigwedge_{(a_1,a_2) \in \kappa_v} f_{\phi(a_1)} \wedge g_{\phi(a_2)}).
\end{equation}
By skew-symmetry this is independent of the order in which we take the wedge over the elements of $\kappa_v$. We see that $s_h(G,\omega,\kappa) = 0$ if $G$ contains a vertex of degree larger than $2 \ell$, as $\bigwedge^n V_{2 \ell}  = \{0\}$ if $n > 2 \ell$.  

The following proposition follows immediately from \cite[Proposition 1]{RS17}.
\begin{proposition}\label{prop:inv keuze}
Let $G$ be an Eulerian graph with Eulerian orientation $\omega$ and compatible local pairing $\kappa$ and let $h=(h^v)_{v\in V}$ with $h^v \in (\bigwedge V_{2 \ell})^*$ for each $v\in V$. 
Then $s_{h}(G,\omega,\kappa)$ is independent of the choice of $\omega$ and $\kappa$ and we write $s_h(G):=s_h(G,\omega,\kappa)$. 
\end{proposition}
The proposition allows us to make the following definition.
\begin{df}\label{df:skewpart}
Let $G=(V,E)$ be a graph.
For $h=(h^v)_{v\in V}$ with $h^v \in (\bigwedge V_{2 \ell})^*$ for each $v\in V$, the partition function $p_h$ of $h$ is defined as 
\begin{equation}\label{eq:skewpart}
p_h(G):=\left\{\begin{array}{rl} s_h(G) & \text{ if }  G \text{ is Eulerian},
\\ 0 & \text{ otherwise.}
\end{array}\right.
\end{equation}
We note that if $h^v$ only depends on $\deg(v)$ and not on the actual vertex $v$, then $p_h$ is a graph invariant, which, following \cite{RS17}, we call a \emph{skew partition function}.
We moreover note that $p_h(\bigcirc)=-2\ell$.
\end{df}

We now turn to the definition of mixed partition functions.
\subsection{Mixed partition functions}
For $k,\ell \in \N$, let $V_k$ be a vector space of dimension $k$ over $\C$ with basis $\{e_1,\ldots,e_k\}$ and let $V_{2 \ell}$ be a vector space of dimension $2\ell$ over $\C$ with basis $\{f_1,\ldots,f_{2\ell} \}$. We call any $h \in (S V_k \otimes \bigwedge V_{2\ell})^*$ a $(k, 2 \ell)$-\emph{color edge coloring model}\index{edge coloring model!mixed}. 

For a graph $G = (V,E)$ and $H \subseteq E$, the subgraph $(V,H)$ of $G$ is denoted by $G(H)$. If $G(H)$ is Eulerian, then we say that $H$ is Eulerian. For $F \subseteq E$ Eulerian, let $\omega$ be an Eulerian orientation of the edges of $G(F)$ and let $\kappa$ be a compatible local pairing of $G(F)$. 
For $v \in V$, let $\delta_{E \setminus F}(v)$ be the set of edges incident with $v$ that are not in $F$. 
For $h \in (S V_k \otimes \bigwedge V_{2\ell})^*$, we now define 
\begin{equation}\label{eq:partialmixed}
s_{h}(G,F,\omega, \kappa):=(-1)^{c(\kappa)} \sum_{\substack{\phi:F \rightarrow [2\ell] \\ \psi:E \setminus F \rightarrow [k]}} \prod_{v \in V} h ( \bigodot_{a \in \delta_{E \setminus F}(v)} e_{\psi(a)} \otimes \bigwedge_{(a_1,a_2) \in \kappa_v} f_{\phi(a_1)} \wedge g_{\phi(a_2)}).
\end{equation}
Fix one coloring $\psi : E \setminus F\to [k]$ and define for $v\in V$, $h_\psi^{v}\in (\bigwedge V_{2\ell})^*$ by 
\[h_\psi^v(f_{i_1}\wedge\cdots \wedge f_{i_{\deg_F(v)}}):=h( \bigodot_{a \in \delta_{E \setminus F}(v)} e_{\psi(a)}\otimes f_{i_1}\wedge\cdots \wedge f_{i_{\deg_F(v)}}) \]
 for $i_1,\ldots,i_{\deg_F(v)}\in [2\ell]$ and let $h^v_\psi$ evaluate to zero on $\bigwedge^{d}V_{2\ell}$ for all $d\neq \deg_F(v)$.
 Then we observe that 
 \[
 s_{h}(G,F,\omega, \kappa)=\sum_{\psi:E\setminus F\to [k]}p_{h_\psi}(G(F)),
 \]
 and hence $s_{h}(G,F,\omega, \kappa)$ does not depend on the choice of $\omega$ and $\kappa$ and we may define $s_{h}(G,F):=s_h(G,F,\omega, \kappa)$ for any choice of Eulerian orientation $\omega$ and compatible labelling~$\kappa$.
\begin{df}
For $h \in (S V_k \otimes \bigwedge V_{2\ell})^*$ the partition function of $h$ is defined by, for a graph $G = (V,E)$,
\begin{equation}\label{eq:mixed}
p_h(G) := \sum_{\substack{ F \subseteq E \\ F \text{ Eulerian}}} s_{h}(G,F). 
\end{equation}
We sometimes refer to the partition function just defined as a \emph{mixed partition function}\index{partition function!mixed} so as to distinguish it from ordinary partition functions and skew partition functions. 
We note that $p_h(\bigcirc) = k-2\ell$. 
\end{df}
We can now state our main theorem concerning the edge-rank connectivity of mixed partition function.

\begin{theorem}\label{thm:finite erc}
If $f$ is the partition function of an element $h \in (S V_k \otimes \bigwedge V_{2\ell})^*$, then 
\begin{equation}
\rk (M_{f,t} )\leq (k+2 \ell)^t \text{ for each } t \in \N.
\end{equation}
\end{theorem}
We will prove this theorem in Section \ref{sec:rankgrowth}. 
First we give several examples of mixed partition functions in the next section.
\section{Examples of mixed partition functions}\label{sec:examples}
We start with some basic examples.
\subsection{Basic examples}
\begin{ex}
Let $h$ be a $(k,0)$-color edge coloring model. For a graph $G = (V,E)$ and $F \subseteq E$, we have that $s_{h}(G,F) = 0$ unless $F = \emptyset$. So $p_h(G)$ reduces to 
\[
s_{h}(G,\emptyset) = \sum_{\psi: E \rightarrow [k]} \prod_{v \in V} h(\bigodot_{a \in \delta(v)} e_{\psi(a)}).
\]
So we see that $p_h$ is an ordinary partition function as in (\ref{eq:pf}). We similarly see that if $h$ is a $(0,2\ell)$-color edge coloring model that $p_h$ reduces to $s_h(G,E)$. So $p_h$ is a skew partition function as in (\ref{eq:skewpart}). 
\end{ex}

\begin{ex}
If $h_0 \in (S V_k)^*$ and $h_1 \in (\bigwedge V_{2\ell})^*$, then let $h = h_0 \otimes h_1 \in (S V_k \otimes \bigwedge V_{2\ell})^*$. For a graph $G = (V,E)$ and $F \subseteq E$ Eulerian, it follows directly from (\ref{eq:partialmixed}) that $s_{h}(G,F) = p_{h_0}(G(E\setminus F)) p_{h_1}(G(F))$. So we find that 
\begin{equation}\label{eq:mixedtensor}
p_h(G) = \sum_{\substack{F \subseteq E \\ F \text{ Eulerian}}} p_{h_0}(G(E\setminus F)) p_{h_1}(G(F)). 
\end{equation}
\end{ex}

We now move to a few more involved examples of mixed partition functions. 

\subsection{The characteristic polynomial}\label{sec:spin}
In this subsection we assume that our graphs do not have the circle, $\bigcirc$, as a connected component. 
The \emph{adjacency matrix} $A$ of $G$ is the $V \times V$ matrix such that for $i,j \in V$ with $i \neq j$, $A(i,j)$ is the multiplicity of the edge $\{i,j\}$ in $E$ and such that for $i \in V$, $A(i,i)$ is twice the number of loops at the vertex $i$. 
The \emph{characteristic polynomial} $p(G)$ of $G$ is defined as $p(G;t): = \text{det}(tI-A)$. 
We show below that for each $t$ there exists an $(2,2)$-edge-coloring model $h(t)$ such that $p(G;t)=p_{h(t)}(G)$ for all graphs $G$, see Proposition~\ref{prop:char pol} below.
First we turn to a question of de la Harpe and Jones concerning spin models and the characteristic polynomial.
To this end give the definition of the partition function of a spin model. We follow the definition of de la Harpe and Jones \cite{HJ}. 

Let $n \in \N$ and let $B$ be a symmetric $n \times n$ matrix with values in some commutative ring $R$. We call $B$ a \emph{spin model over $R$}. 
The partition function, $p_B$, of the spin model $B$ is defined, for a graph $G = (V,E)$, by
\begin{equation}\label{eq:vertexmodel}
p_{B}(G) := \sum_{\kappa:V \rightarrow [n]} \prod_{\{v_1,v_2\} \in E} B(\kappa(v_1),\kappa(v_2)).
\end{equation}
Note that this is well-defined as $B$ is a symmetric matrix. 

De la Harpe and Jones~\cite[Problem 1]{HJ} asked about the existence of a spin model $B$ over $\C[t]$, such that $p_{B}(G) = p(G;t)$ for each graph $G$. 
In the following proposition we shall show that the answer to this question is negative.
In fact, we show something stronger. 

\begin{proposition}\label{prop:no ecm}
There does not exist an edge coloring model $h$ such that $p_h(G) = p(G;0)$ for all graphs $G$.
\end{proposition}
This proposition is indeed stronger than we need, since, by a result of Szegedy~\cite{Sz7}, the partition function of any spin model is equal to the partition function of an ordinary edge coloring model and hence Proposition~\ref{prop:no ecm} rules out the existence of a spin model of which the partition function equals the characteristic polynomial evaluated at $0$ and therefore provides a negative answer to the question of de la Harpe and Jones.
Proposition~\ref{prop:char pol} below may serve as an alternative answer to the question of de la Harpe and Jones.

We now turn to a proof of Proposition~\ref{prop:no ecm}.
\begin{proof}[Proof of Proposition~\ref{prop:no ecm}]
Let us abuse notation and write $\det(G)$ for the determinant of the adjacency matrix of $G$.
Note that for a graph with an even number of vertices we have $p(G;0)=\det(G)$.
We will make use of the characterization of partition functions of edge coloring models as given in \cite{DGLRS12}.
Fix $k$ and consider the graph $G$ consisting of $k+1$ copies of the $6$-cycle $C_6$.
Direct one edge in each cycle and label the endpoints of these arcs $1$ up to $k+1$. For a permutation $\pi\in S_{k+1}$, denote by $G_\pi$ the graph obtained from $G$ by letting $\pi$ permute the endpoints of the directed edges.
Note that if the permutation $\pi$ can be written as the product of disjoint cycles $\pi_1,\dots,\pi_t$, then $G_\pi$ is the graph consisting of $t$ cycles, of length $6|\pi_1|, \dots, 6|\pi_t|$ respectively. 
Here $|\pi_i|$ denotes the length of the cycle $\pi_i$; we include cycles of length $1$.
If $p(G;0)$ is the partition function of a $k$-color edge coloring model, then, by \cite[Theorem 1]{DGLRS12}, it must satisfy
\begin{equation}\label{eq:constraint ecm}
\sum_{\pi\in S_{k+1}} \text{sgn}(\pi) p(G_\pi;0)=0.
\end{equation}
It follows from, for example, \cite[Section 1.4.3]{BH}, that $\det(C_k)=0$ if $k=0\mod 4$ and $\det(C_k)<0$ if $k=2 \mod 4$.
This implies that for $\det(G_\pi)$ to be non-zero none of the cycles $\pi_1, \dots, \pi_t$ may be of even length. However, if all cycles in the cycle decomposition of $\pi$ are of odd length, then the parity of the number of these cycles is equal to the parity of $k+1$.
So in this case, $\det(G_\pi)$ is strictly positive if this parity is even and strictly negative if this parity is odd for all such permutations $\pi$. As all cycles of $\pi$ are odd we have $\text{sgn}(\pi) = 1$. 
Since $\det(G_\pi)=p(G_\pi;0)$ for all $\pi$, this shows that $\sum_{\pi\in S_{k+1}} \text{sgn}(\pi) p(G_\pi;0)$ is either strictly positive or strictly negative. So it is non-zero.
So we conclude that \eqref{eq:constraint ecm} is violated and hence that $p(\cdot;0)$ cannot be the partition function of any edge coloring model.
\end{proof}

\begin{proposition}\label{prop:char pol}
For each $t \in \C$, there exists a $(2,2)$-color edge coloring model $h(t)$ such that $p_{h(t)}(G) = p(G;t)$ for all graphs $G$.
\end{proposition}
\begin{proof}
Using the Leibniz expansion of the determinant, Sachs \cite{S67} gave an expression of the characteristic polynomial of a graph $G$ in terms of certain subgraphs of $G$. The expression extends to graphs with multiple edges and loops. Let $G = (V,E)$ be a graph. Let $\mathcal{H}$ be the set of $H \subseteq E$ such that each connected component of $G(H)$ is either a vertex, an edge or a cycle. For $H \in \mathcal{H}$, let $e^*(H)$ and $c(H)$ denote the number of connected components of $G(H)$ that are edges and cycles respectively. Let $V[H] \subseteq V$ be the set of vertices of $G$ that are incident with an edge of $H$. Then Sachs showed that 
\begin{equation}\label{charpoly}
p(G;t) = \sum_{H \in \mathcal{H}} (-1)^{e^*(H)} (-2)^{c(H)} t^{|V|-|V[H]|}.
\end{equation}
We now give a $(2,2)$-color edge coloring model $h=h(t)$ such that $p_h(G) = p(G;t)$ for each $t\in \C$ and graph $G$. Let $h$ be defined as follows:
\begin{align*}
&h(e_1^{\odot i} \otimes f_1 \wedge g_1) = 1 \text{ for } i \in \N, \\
&h(e_1^{\odot i} \odot e_2) = \sqrt{-1} \text{ for } i \in \N, \\
&h(e_1^{\odot i}) = t \text{ for } i \in \N,
\end{align*}
and let $h$ evaluate to $0$ on basis elements of $S V_2 \otimes \bigwedge V_{2}$ that are not in the span of these elements. 
Now let $F \subseteq E$ be Eulerian. We compute $s_h(G,F)$. If $G(F)$ has a vertex that is not of degree $0$ or $2$, then $s_h(G,F) = 0$. 

So let us assume that each vertex of $G(F)$ has degree $0$ or $2$. 
Let $\omega$ be an Eulerian orientation of $F$ with a compatible local pairing $\kappa$ of $G(F)$. Now let $\phi: F \rightarrow [2]$ and $\psi : E \setminus F \rightarrow [2]$. We first note that for the contribution of $\phi$ and $\psi$ to $s_{h}(G,F,\omega,\kappa)$ to be non-zero, we need $\psi^{-1}(2)$ to be a matching in $G$ that is not incident with any edge in $F$. 

Now fix $\psi: E \setminus F \rightarrow [2]$ such that $\psi^{-1}(2)$ is a matching in $G$ that is not incident with any edge in $F$. Let $H = F \cup \psi^{-1}(2) \subseteq E$. Note that at each vertex $v \in V$ that is not incident with $H$, we see $h(e_1^{\odot d(v)}) = t$. There are $|V|-|V[H]|$ such vertices $v$. If $v,u \in V$ are two vertices such that $\{u,v\}$ is an isolated edge of $(V,H)$, then at $u$ we see $h(e_1^{\odot d(u)-1} \odot e_2) = \sqrt{-1}$ and at $v$ we see $h(e_1^{\odot d(v)-1} \odot e_2) = \sqrt{-1}$. So two vertices $u,v$ such that $\{u,v\}$ is an isolated edge of $(V,H)$ contribute $-1$ to the partition function. 

Now consider the colorings $\phi: F \rightarrow [2]$. Such a $\phi$ gives a non-zero contribution if and only if it is constant on the edges of each $2$-regular connected component of $(V,F)$. So there are exactly $2^{c(F)}$ colorings $\phi: F \rightarrow [2]$ that have a non-zero contribution. 
As $c(\kappa)=c(F)=c(H)$, we see that the total contribution to the partition function of these colorings is exactly
\begin{equation}\label{eq:char eq}
 (-1)^{e^*(H)} (-2)^{c(H)} t^{|V|-|V[H]|}.
\end{equation}
Now summing over all $F$ and corresponding $\phi$ and $\psi$, we find that $p_h(G)$ is indeed equal to $p(G;t)$ by (\ref{charpoly}). 
\end{proof}

\subsection{Evaluations of the circuit partition polynomial}
The circuit partition polynomial, introduced, in a slightly different form, by Martin in his thesis~\cite{Martin}, is related to Eulerian walks in graphs and to the Tutte polynomial of planar graphs. 
Several identities for the circuit partition polynomial were established by Bollob\'as~\cite{Bol} and Ellis-Monaghan~\cite{EM}.

We say that two circuits are equivalent if one can be obtained from the other by possibly changing the starting vertex or the direction of the walk. For a graph $G = (V,E)$, let $X(G)$ be a set of representatives of this equivalence relation. Let $\mathcal{C}(G)$ be the collection of all partitions of $E$ into circuits in $X(G)$. For $C \in \mathcal{C}(G)$, let $|C|$ be the number of circuits in the partition. 

The circuit partition polynomial $J(G,x)$ is defined for a graph $G$ by
\[
J(G,x) := \sum_{C \in \mathcal{C}(G)} x^{|C|}. 
\]
So if $G$ is not an Eulerian graph, then $J(G,x) = 0$. 
We clearly have that $J(G\cup H,x) = J(G,x) J(H,x)$ for two graphs $G$ and $H$ and it is natural to define $J(\bigcirc,x) = x$. 
We shall show that every integer evaluation of the circuit partition polynomial can be expressed as a mixed partition function.
We first recall the result from \cite{RS17} that positive and negative even integer evaluations can be expressed as ordinary partition functions and skew partition functions respectively, after which we show that odd negative evaluations can be expressed as mixed partition functions (with both $k$ and $2\ell$ positive).

For $k \in \N$, it was shown in \cite{Bol,EM} that $J(G,k)$ can be expressed as
\begin{equation}
J(G,k) = \sum_{A} \prod_{v \in V} \prod_{i=1}^k (\text{deg}_{A_i}(v) - 1)!!,\label{eq:cyc part pos}
\end{equation}
where $A$ ranges over ordered partitions of $E$ into $k$ subsets $A_1,\ldots,A_k$ such that $A_i$ is Eulerian for all $i \in [k]$. We use the convention that for a nonnegative integer $d$, $(2d-1)!!=(2d-1)\cdot(2d-3)\cdots 3\cdot 1$ and $(2d)!!=0$.

We express \eqref{eq:cyc part pos} as the partition function of $h_0 \in (S V_k)^*$ as follows. 
For $(\alpha_1,\ldots,\alpha_k) \in \N^k$, we set 
\begin{equation}\label{eq:hzero}
h_0(\bigodot_{i \in [k]} e_i^{\odot \alpha_i}):=\prod_{i=1}^k (\alpha_i-1)!!.
\end{equation}
Using \eqref{eq:cyc part pos} it is not difficult to see that $p_{h_0}(G)=J(G,k)$ for each graph $G$, cf. \cite{RS17}. 

Bollob\'as \cite{Bol} showed that the evaluation of the circuit partition polynomial $J(G,x)$ of a graph at negative even integers $-2 \ell$ can be expressed as
\begin{equation} \label{eq:bolpol}
J(G,-2 \ell)=\sum_{H_1,\ldots,H_\ell}(-2)^{\sum_{i=1}^\ell c(H_i)},
\end{equation}
where this sum runs over all ordered partitions $H_1,\ldots,H_\ell$ of the edge set of $G$ such that for each $i \in [\ell]$ each vertex in $(V,H_i)$ has degree $0$ or degree $2$ and where $c(H_i)$ is the number of $2$-regular connected components of $(V,H_i)$. Now let $h_1 \in (\bigwedge V_{2\ell})^*$ be defined, for $S \subseteq [ \ell]$, by 
\begin{equation}\label{eq:cycletensor}
h_1( \bigwedge_{i \in S} f_i \wedge g_i) = 1;
\end{equation}
$h_1$ evaluates to zero on all other basis elements not in the span of the elements above.
In \cite{RS17} it is shown, using \eqref{eq:bolpol}, that $p_{h_1}(G) = J(G,-2\ell)$. 

We will next show that mixed partition functions can also express evaluations of the circuit partition polynomial at negative odd integers.
In \cite{EM}, Ellis-Monaghan showed for a graph $G = (V,E)$ that
\begin{equation} \label{eq:EMpol}
J(G,x+y) = \sum_{A \subseteq E} J(G(A),x)J(G({E \setminus A}), y).
\end{equation}
Now, for a negative odd integer $-2\ell+1$, let $h_0 \in (S V_1)^*$ correspond to $k=1$ in (\ref{eq:hzero}) and let $h_1 \in (\bigwedge V_{2\ell})^*$ be as in (\ref{eq:cycletensor}). Let $h = h_0 \otimes h_1 \in (S V_1 \otimes \bigwedge V_{2\ell})^*$. Then by (\ref{eq:mixedtensor}) and (\ref{eq:EMpol}) we find that $p_h(G) = J(G,-2\ell+1)$, giving us an expression of $J(G,-2\ell+1)$ as a mixed partition function.

\begin{rem}
Let us finally remark that the cycle partition polynomial evaluated at $x\notin \mathbb{Z}$ cannot be realized as the mixed partition function of any edge-coloring model. This follows from Theorem~\ref{thm:finite erc} and the fact that by \cite[Proposition 2]{S15} the rank of the submatrix of $M_{J(\cdot,x),2n}$, indexed by fragments in which each component consists of an unlabeled vertex with two open ends incident with it, has rank at least $n!$, cf.~\cite{RS17}.
\end{rem}
\section{The edge-rank connectivity of mixed partition functions}\label{sec:rankgrowth}

In this section we prove Theorem \ref{thm:finite erc}.

We first show a lemma on matchings that will be useful later on. If $M$ and $N$ are directed perfect matchings on the same vertex set, then we denote by $o(M\cup N)$ the parity of the number of arcs in $M\cup N$ that need to be flipped to make $M\cup N$ into an Eulerian digraph. Note that we call a digraph Eulerian if for each vertex its in degree is equal to its out degree. Since each cycle in $M\cup N$ has even length this is well defined. Recall that $c(M\cup N)$ is the number of connected components of $M\cup N$. 

\begin{lemma}\label{matching sign}
Let $M$ and $N$ be two directed perfect matchings on $[2m]$ for $m \in \N$. Then the sign of any permutation in $S_{2m}$ that sends $N$ to $M$ is equal to $(-1)^{c(M\cup N) + o(M\cup N)}$. 
\end{lemma}
\begin{proof}
Note that all permutations that send $M$ to $N$ have the same sign, as all permutations in $S_{2m}$ that stabilize $M$ have trivial sign. We may assume that $M\cup N$ consists of a single connected component.
Let $\sigma_1,\sigma_2 \in S_{2m}$ be permutations that flip arcs of $M$ and $N$ respectively, such that $\sigma_1 M\cup \sigma_2 N$ has an Eulerian orientation. If the vertices of the cycle are given by $v_1,v_2,\ldots,v_{2m}$ in cyclic order, then the permutation $\tau = (v_1 v_2 \dots v_{2m})$ has the property that $\tau \sigma_1 M = \sigma_2 N$. 
As $2m$ is even, the sign of $\tau$ is $-1$. So the permutation $\sigma_2^{-1} \tau \sigma_1$ sends $M$ to $N$. Per construction we have that $\text{sgn}(\sigma_1)\text{sgn}(\sigma_2) = (-1)^{o(M\cup N)}$. This proves the lemma. 
\end{proof}
In the proof we will make use of some linear algebra that we will now define. Let $(\cdot,\cdot)$ be the nondegenerate symmetric bilinear form on $V_k=\C^k$ given by $(x,y):=x^Ty$ for $x,y\in \C^k$.
Let $\langle\cdot,\cdot \rangle$ be the nondegenerate skew-symmetric bilinear form on $V_{2 \ell}=\C^{2\ell}$ given by $\la x,y\ra=x^TJy$ for $x,y\in \C^{2\ell}$,  where $J$ is the $2\ell\times 2\ell$ matrix 
 $\left (\begin{array}{rr} 0&I\\-I&0\end{array}\right)$, with $I$ the $\ell\times \ell$ identity matrix.
We define  $V_{k,2\ell}:=V_k\oplus V_{2\ell}$. We write an element $w$ of $V_{k,2\ell}$ as $w_{\bar{0}} + w_{\bar{1}}$, where $w_{\bar{0}} \in V_k$ and $w_{\bar{1}} \in V_{2\ell}$. We equip this space with a nondegenerate bilinear form $[\cdot,\cdot]$ defined by 
\begin{equation}\label{eq:bilform}
[x,y]:=(x_{\bar 0},y_{\bar 0})+\langle x_{\bar 1},y_{\bar 1}\rangle, 
\end{equation}
for $x,y\in V_{k,2\ell}$.
We note that this form is often called a \emph{super symmetric bilinear} form, cf. \cite{CW12}.

\vspace{5mm}

\emph{Proof of Theorem \ref{thm:finite erc}.} Our goal is to show that for each $t \in \N$, we can write $M_{f,t}$ as a Gram matrix of vectors in $V_{k,2\ell}^{\otimes t}$ with respect to the bilinear form $[ \cdot , \cdot ]$. This then immediately implies Theorem~\ref{thm:finite erc}.

Let $t \in \N$ and let $F = (V,E)$ be a $t$-fragment. Recall that a $t$-fragment is a graph with $t$ vertices of degree $1$ labeled $1,\dots,t$. The set of unlabeled vertices of $F$ is denoted by $V'(F)$.
A subset $H \subseteq E$ is called \emph{Eulerian}\index{fragment!Eulerian} if the degree of each unlabeled vertex in $F(H)$ is even. Let $H \subseteq E$ be Eulerian. Let $S(H)$ be the set of labeled vertices incident with an edge in $H$. If $H$ is chosen, we refer to $S(H)$ as $S$. Note that $|S|$ is even because $H$ is Eulerian. We identify the labeled vertices with $[t]$ according to the labeling. Through this identification we view $S$ as a subset of $[t]$. 

We now extend some of the definitions we gave for graphs to fragments.
An \emph{Eulerian orientation}\index{fragment!Eulerian} $\omega$ of $H$ is an orientation of the edges of $H$ such that in $F(H)$, at each unlabeled vertex the number of incoming arcs is equal to the number of outgoing arcs. A \emph{local pairing}\index{fragment!local pairing} $\kappa$ of $F(H)$ is an assignment $\kappa$ to each $v \in V'(F)$ of a decomposition $\kappa_v$ of the edges in $H$ incident with $v$ into ordered pairs. The local pairing $\kappa$ is called \emph{compatible}\index{fragment!compatible} with a Eulerian orientation $\omega$ if for each $v \in V'(F)$ and for each $(a_1,a_2) \in \kappa_v$ the arc $a_1$ is incoming under $\omega$ and the arc $a_2$ is outgoing under $\omega$.

Now let $\kappa$ be a local pairing of $F(H)$ compatible with an Eulerian orientation $\omega$ of $H$. Note that $\kappa$ partitions the edge set of $H$ into circuits and directed trails that begin and end in labeled vertices. We call this decomposition the $\kappa$-\emph{decomposition} of $H$. Let $\hat{c}(\kappa)$ be the number of circuits in the $\kappa$-decomposition. Define $M(\omega,\kappa)$ to be the directed perfect matching on $S \subseteq [t]$ such that $(i,j)$ is an arc of $M(\omega,\kappa)$ if there is a directed trail in the $\kappa$-decomposition from $i$ to $j$. Write $S = \{i_1,\dots,i_{|S|}\}$ with $i_1 < \dots < i_{|S|}$. 
Let $\text{sgn}(M(\omega,\kappa))$ be the sign of a permutation that sends $M(\omega,\kappa)$ to the directed perfect matching with arcs $(i_1,i_2),\dots,(i_{|S|-1}, i_{|S|})$. This is well-defined by Lemma~\ref{matching sign}.  

Let $\chi = (\chi_0,\chi_1)$ with $\chi_0 : [t] \setminus S \rightarrow [k]$ and $\chi_1 : S \rightarrow [2\ell]$. Such a pair $\chi = (\chi_0,\chi_1)$ is called \emph{consistent} with $S$. We say that a coloring $\psi: E \setminus H \rightarrow [k]$ \emph{extends} $\chi_0$ if, for each $i \in [t] \setminus S$, we have $\chi_0(i) = \psi(a)$, where $a \in E \setminus H$ is the edge incident with $i$. We denote this by $\psi \sim \chi_0$. Similarly, we say that $\phi: H \rightarrow [2\ell]$ extends $\chi_1$ if, for each $i \in S$, we have $\chi_1(i) = \phi(a)$, where $a \in H$ is the edge incident with $i$. Again, we denote this by $\phi \sim \chi_1$. 

For $i \in [t] \setminus S$, let $c_{\chi,\omega,i} = e_{\chi_0(i)}$, and for $i \in S$, let $c_{\chi,\omega,i} = f_{\chi_1(i)}$ if the edge incident with $i$ is incoming at $i$ under $\omega$ and let $c_{\chi,\omega,i} = g_{\chi_1(i)}$ if the edge incident with $i$ is outgoing at $i$ under $\omega$. 
We define the tensor $t_{h,\chi}'(F,H,\omega,\kappa)$ in $V_{k,2\ell}^{\otimes t}$ by
\begin{align*}
& t_{h,\chi}'(F,H,\omega,\kappa) := \\
&(-1)^{\hat{c}(\kappa)}\sum_{\substack{\psi \sim \chi_0 \\  \phi \sim \chi_1}}\prod_{v \in V'(F)} h(\bigodot_{a \in \delta_{E\setminus H}(v)} e_{\psi(a)} \otimes \bigwedge_{(a_1,a_2) \in \kappa_v} f_{\phi(a_1)} \wedge g_{\phi(a_2)}) \bigotimes_{i \in [t]} c_{\chi,\omega,i}, 
\end{align*}
where the sum runs over all $\psi: E \setminus H \rightarrow [k]$ with $\psi \sim \chi_0$ and all $\phi : H \rightarrow [2\ell]$ with $\phi \sim \chi_1$. 
We define
\[
t_h'(F,H,\omega,\kappa) := \sum_{\substack{\chi \text{ consistent} \\ \text{ with } S}} t_{h,\chi}'(F,H,\omega,\kappa),
\]
and finally we define
\[
 t_{h}(F,H,\omega,\kappa) := (-1)^{|S|/4}\text{sgn}(M(\omega,\kappa))t_h'(F,H,\omega,\kappa).
\]
We first make an important observation. Let $\omega'$ be obtained from $\omega$ by inverting the arcs in a directed trail $P$ in the $\kappa$-decomposition and let $\kappa'$ be obtained from $\kappa$ by inverting all the pairings in the directed trail $P$ (hence $\kappa'$ is compatible with $\omega'$). Note that $\text{sgn}(M(\omega,\kappa)) = -\text{sgn}(M(\omega',\kappa'))$, as $M(\omega',\kappa')$ is obtained from $M(\omega,\kappa)$ by inverting the direction of one arc. The total number of pairings and arcs in the directed trail $P$ is odd. So similar to what we have seen in the proof of~\cite[Lemma 2]{RS17}, we find that $t_h'(F,H,\omega,\kappa) = - t_h'(F,H,\omega',\kappa')$. This shows that 
\begin{equation}\label{eq:pathinvariant}
 t_h(F,H,\omega,\kappa) =  t_h(F,H,\omega',\kappa'). 
\end{equation}
Now let $F_1 = (V_1,E_1)$ and $F_2=(V_2,E_2)$ be two $t$-fragments with Eulerian subsets $H_1 \subseteq E_1$ and $H_2 \subseteq E_2$ such that $S(H_1) = S(H_2) = S$. Let $G = (V,E) = F_1*F_2$. Note that $H_1$ and $H_2$ induce an Eulerian subset of $E$. We denote this set by $H_1*H_2$. For $i = 1,2$, let $\omega_i$ be an Eulerian orientation of $H_i$ with a compatible local pairing $\kappa_i$ of $F_i(H_i)$. We next show that 
\begin{equation}\label{eq:tensortos}
[t_h(F_1,H_1,\omega_1,\kappa_1), t_h(F_2,H_2,\omega_2,\kappa_2)] = s_h(G,H_1*H_2).
\end{equation}
By (\ref{eq:pathinvariant}) we may assume that $\omega_1,\kappa_1, \omega_2$ and $\kappa_2$ are chosen in such a way that $(S,M(\omega_1,\kappa_1) \cup M(\omega_2,\kappa_2))$ is an Eulerian digraph. By Lemma~\ref{matching sign} we see that 
\[
\text{sgn}(M(\omega_1,\kappa_1))\text{sgn}(M(\omega_2,\kappa_2))=(-1)^{c(M(\omega_1,\kappa_1) \cup M(\omega_2,\kappa_2))},
\]
as $o(M(\omega_1,\kappa_1) \cup M(\omega_2,\kappa_2)) = 0$. Furthermore, $\omega_1$ and $\omega_2$ induce an Eulerian orientation $\omega$ of $H_1 * H_2$ and the local pairing $\kappa$ of $G(H_1 *H_2)$ induced by $\kappa_1$ and $\kappa_2$ is compatible with $\omega$. So we find that
\begin{equation}\label{eq:signscircuits}
\text{sgn}(M(\omega_1,\kappa_1))\text{sgn}(M(\omega_2,\kappa_2)) (-1)^{\hat{c}(\kappa_1)}(-1)^{\hat{c}(\kappa_2)} = (-1)^{c(\kappa)}. 
\end{equation}

Now let $\chi = (\chi_0,\chi_1)$ and $\chi' = (\chi_0',\chi_1')$ both be consistent with $S$. We consider
\begin{equation}\label{eq:tensorfirst}
[t_{h,\chi}'(F_1,H_1,\omega_1,\kappa_1),t_{h,\chi'}'(F_2,H_2,\omega_2,\kappa_2)].
\end{equation}
Note that this is equal to $0$ if $\chi_0$ and $\chi_0'$ do not agree. Furthermore, as the orientations of $\omega_1$ and $\omega_2$ are opposite at a labeled vertex in $S$, we see that $\chi_1$ and $\chi_1'$ also have to agree for (\ref{eq:tensorfirst}) to be non-zero. So let us assume that $\chi = \chi'$. Note that as the orientation $\omega$ is Eulerian, at half of the vertices in $S$ the arc of $H_1$ is incoming and the arc of $H_2$ is outgoing. So at such a vertex $i$ the bilinear form becomes $\langle f_{\chi_1(i)}, g_{\chi_1(i)} \rangle = -1$. At the other half of the vertices in $S$ the arc of $H_2$ is incoming and the arc of $H_1$ is outgoing. So at such a vertex $i$ the bilinear form becomes $\langle g_{\chi_1(i)}, f_{\chi_1(i)} \rangle = 1$. These contributions cancel with $(-1)^{|S(H_1)|/4}(-1)^{|S(H_2)|/4}$. Together with (\ref{eq:signscircuits}) this shows (\ref{eq:tensortos}).

Now, for $i=1,2$, let $H_i \subseteq E_i$ and let $\omega_i$ be an Eulerian orientation of $H_i$ with a compatible local pairing $\kappa_i$ of $F_i(H_i)$. Suppose that $S(H_1) \neq S(H_2)$. Then it follows that 
\begin{equation}\label{eq:tensorthird}
[t_{h}(F_1,H_1,\omega_1,\kappa_1), t_{h}(F_2,H_2,\omega_2,\kappa_2)] = 0,
\end{equation}
because at $i$ in the symmetric difference of $S(H_1)$ and $S(H_2)$ there occurs an element of $V_k$ at one side of the bilinear form and an element of $V_{2\ell}$ at the other side. 
 
Note that as $H_1$ and $H_2$ run over all Eulerian subsets of $F_1$ and $F_2$, we have that $H_1 *H_2$ runs over all Eulerian subsets of $G$. So it follows from (\ref{eq:tensortos}) and (\ref{eq:tensorthird}) that
\begin{align}
\biggl[&\sum_{\substack{H_1 \subseteq E_1 \\ H_1 \text{ Eulerian}}}t_h(F_1,H_1,\omega_1,\kappa_1),\sum_{\substack{H_2 \subseteq E_2 \\ H_2 \text{ Eulerian}}} t_h(F_2,H_2,\omega_2,\kappa_2)\biggr] = \\&  \sum_{\substack{H \subseteq E \\ H \text{ Eulerian}}} s_h(G, H,\omega,\kappa) = f(G),
\end{align}
where, for $i=1,2$, $\kappa_i$ is a local pairing of $F_i(H_i)$ compatible with an Eulerian orientation $\omega_i$ of $H_i$. This shows that $M_{f,t}$ indeed is the Gram matrix of a set of vectors in $V_{k,2\ell}^{\otimes t}$ with respect to the bilinear form $[ \cdot , \cdot ]$. So the rank of $M_{f,t}$ is bounded by $(k+2\ell)^t$. This proves Theorem~\ref{thm:finite erc}. \qed

\section{Connections to invariant theory}\label{sec:invariant}
In this section we will indicate how one might prove the conjectured converse to Theorem~\ref{thm:finite erc} by saying how mixed partition functions connect to the invariant theory of the orthogonal and symplectic groups and the orthosymplectic Lie superalgebra.

Following~\cite{DGLRS12,S15,RS17}, to prove a converse to Theorem~\ref{thm:finite erc} one essentially needs to prove two statements. First one needs an algebraic characterization of mixed partition functions Second one needs to show that a multiplicative graph parameter with finite edge-rank connectivity satisfies these algebraic conditions. 
Below we comment on how such a possible algebraic characterization is deeply connected to invariant theory.

For $k,\ell \in \N$, recall that $V_k$ is a vector space of dimension $k$ over $\C$ with basis $\{e_1,\ldots,e_k\}$ and recall that $V_{2\ell}$ is a vector space over $\C$ of dimension $2 \ell$ with basis $\{f_1,\ldots,f_{2\ell}\}$. 
We furthermore defined $V_{k,2 \ell}$ as $V_k \oplus V_{2\ell}$. 
The \emph{orthogonal group} $\orth_{k}$ is the group of $k\times k$ matrices that preserve the symmetric bilinear form; i.e., for $g\in \C^{k\times k}$, $g\in \orth_{k}$ if and only if $(gx,gy)=(x,y)$ for all $x,y\in V_{k}$.
The \emph{symplectic group} $\spl_{2\ell}$ is the group of $2\ell\times 2\ell$ matrices that preserve the skew-symmetric bilinear form; i.e., for $g\in \C^{2\ell\times 2\ell}$, $g\in \spl_{2 \ell}$ if and only if $\langle gx,gy\rangle =\langle x,y\rangle $ for all $x,y\in V_{2\ell}$.
The group of $(k+2\ell)\times (k+2\ell)$-matrices that preserve the form \eqref{eq:bilform} can be shown to be the direct product of $\orth_{k}$ and $ \spl_{2 \ell}$.

Consider a basis element $f=f_{i_1}\wedge \cdots\wedge f_{i_{t}}\in \bigwedge V_{2\ell}$. We call $f$ \text{even} if $t$ is even and \text{odd} otherwise.
Let us for $x\in \Z/2\Z=\{\overline 0,\overline 1\}$ denote by $(\bigwedge V_{2\ell})_x$ the subspace of $\bigwedge V_{2\ell}$ spanned by the basis elements $f_{i_1}\wedge \cdots\wedge f_{i_{t}}$ for which $x=t\mod 2$.
We define
\[
R=R(V_{k,2 \ell}) := \text{Sym}( S V_k \otimes (\bigwedge V_{2 \ell})_{\overline 0} ). 
\]

Through the canonical isomorphisms $V_{k}\cong (V_{k}^{*})^*$ and $V_{2 \ell} \cong \mathcal (V_{2 \ell}^*)^*$, we can view $R$ as the space of regular functions on $ (S V_k  \otimes  (\bigwedge V_{2 \ell})_{\overline 0})^*$.
By $\C \G$ we denote the space of formal linear combinations of elements of $\G$ with complex coefficients.
Analogous to \cite{DGLRS12} and \cite{RS17}, we can define a map $p:\C\G\to R$ such that for each $(k,2\ell)$-color edge coloring model $h \in \mathcal (S V_k \otimes \bigwedge V_{2 \ell})^*$ we have $p(G)(h)=p_h(G)$ for each graph $G$.  (We refer to \cite{RS18} for the explicit construction.)

To characterize mixed partition functions, following~\cite{DGLRS12,RS17}, the idea is to consider the ideal generated by $p(G)-f(G)$ for all graphs $G$ in $R$, where $f$ is a graph parameter.
Provided certain conditions are satisfied, Hilbert's Nullstellensatz can be used to find a common `zero' for this ideal. Such a common zero $h$ is exactly an edge-coloring model $h$ such that $p_h(G)=p(G)(h)=f(G)$ for all graphs $G$. 

Now in case $\ell=0$, (resp.\ $k=0$), the algebraic characterization of these respective partition function reads that the graph parameter $f$ should map the kernel of the map $p$ to zero (i.e. we first extend $f$ linearly to a map $f:\mathbb{C} \mathcal{G}\to \mathbb{C}$; this map should satisfy $\ker p\subseteq \ker f$).
The image of this map $p$ turns out to be the space of polynomials in $R$ that are invariant under a natural action of the orthogonal group~\cite{DGLRS12} (resp.\ the symplectic group~\cite{RS17}). The kernel of $p$ can be described using the second fundamental theorem of invariant theory for the orthogonal (resp.\ symplectic) group. 
Using reductivity of the orthogonal and symplectic group, it can be shown that this yields a valid characterization of these respective partition functions.

When both $k$ and $2\ell$ are positive, we conjecture that the multiplicative graph parameters $f$ that satisfy $\ker p\subseteq \ker f$ are exactly the mixed partition functions.

The kernel of $p$ can still be described~\cite{RS18} and is related to the second fundamental theorem of invariant theory of the orthosymplectic supergroup cf.~\cite{LZ14}.
So to prove this conjecture along the same lines of \cite{DGLRS12,RS17} we would need a good understanding of the image of the map $p$. 

There is a natural action of $\orth_{k}\times  \spl_{2 \ell}$ on $R$ and the image of the map $p$ consists of $\orth_{k}\times  \spl_{2 \ell}$-invariants. However it is {\bf not} true that the image of $p$ is equal to the space of elements of $R$ that are invariant under the group $\orth_{k}\times  \spl_{2 \ell}$.
To say more about the image of $p$, we need some definitions from supersymmetry. 
We refer to~\cite{CCF11,CW12} for background on supersymmetry and for notation that we use here.

The space $S V_k\otimes \bigwedge V_{2\ell}$ has the structure of a super vector space. 
Its even part is $S V_k\otimes  (\bigwedge V_{2 \ell})_{\overline 0}$ and its odd part is $S V_k\otimes  (\bigwedge V_{2 \ell})_{\overline 1}$. 
The tensor algebra $T=T(S V_k\otimes \bigwedge V_{2\ell})$ then naturally carries the structure of a super algebra. The \emph{super symmetric algebra}, $S$, is the quotient of $T$ by the ideal generated by $x\otimes y-(-1)^{|x||y|}y\otimes x$ with $x,y \in S V_k\otimes \bigwedge V_{2\ell}$ homogeneous elements. (Here $|x|=0$ if $x$ is even and $|x|=1$ if $x$ is odd for any homogeneous element $x$ of a super vector space.)
Then $R$ is a subalgebra of $S$ and we have a natural projection $\Pi:S\to R$.

The orthosymplectic Lie superalgebra $\mathfrak{osp}(V_{k,2\ell})$ is the Lie superalgebra preserving the form (\ref{eq:bilform}), i.e., for each $X \in \mathfrak{osp}(V_{k,2\ell})$, we have that $[Xv,w] =(-1)^{|X||v|} [v,Xw]$, where we assume all elements involved to be homogenous and we view $X$ as an element of $\text{End}(V_{k,2\ell})$. 
(Both $V_{k,2\ell}$ and $\text{End}(V_{k,2\ell})$ naturally carry the structure of a super vector space.)
It turns out that the Lie super algebra $\mathfrak{osp}(V_{k,2\ell})$ and the group $\orth_{k}\times \spl_{2\ell}$ have a natural action on $S$ and that 
\[
\text{im}(p) = \Pi(S^{\orth_{k}\times \spl_{2\ell},\mathfrak{osp}(V_{k,2\ell})}),
\]
the projection onto $R$ of the space of elements of $S$ that are invariant under the action of $\orth_{k}\times \spl_{2\ell}$ and the action of $\mathfrak{osp}(V_{k,2\ell})$. (This is proved in~\cite{RS18} using results of Berele and Regev~\cite{BR87} and recent results of Leher and Zhang~\cite{LZ17}.)

Since the Lie super algebra $\mathfrak{osp}(V_{k,2\ell})$ is not reductive, it is not clear how to continue the proof outline from~\cite{DGLRS12,RS17}. We expect that with a thorough understanding of invariant theory in this super symmetric setting one may finish the proof outline sketched here. 

\section{Concluding remarks}
In this paper we have introduced mixed partition functions and given several examples of mixed partition functions, shown that they have finite edge-rank connectivity and briefly discussed connections with invariant theory.
In~\cite{RS18} we will focus on the algebraic and invariant-theoretic aspects of mixed partition functions. By exploiting recent developments in the invariant theory of the orthosymplectic supergroup~\cite{LZ14,LZ17,LZ17a}, we aim to prove that any multiplicative graph parameter with finite edge-connection rank is a mixed partition function.

It would be interesting to find more examples of mixed partition functions. 
Given the fact that supersymmetry originated in physics, it would be interesting to explore if there are natural statistical physics models that give rise to interesting combinatorial parameters, similar to how the Potts model partition function is related to the Tutte polynomial for example.

\section*{Acknowledgements}
The research leading to these results has received funding from the European Research Council
under the European Union's Seventh Framework Programme (FP7/2007-2013) / ERC grant agreement n$\mbox{}^{\circ}$ 339109. 

We thank Lex Schrijver for useful comments on an earlier version of this paper. We moreover thank the anonymous referees for their constructive feedback.

\end{document}